\documentclass[]{article}

\usepackage[a4paper,top=2.1cm,bottom=2.60cm,left=2.1cm,right=2.1cm]{geometry} 

\usepackage{latexsym}
\usepackage{amsmath}
\usepackage{amsfonts}
\usepackage{mathrsfs}
\usepackage{amstext}
\usepackage{amssymb}
\usepackage{amsthm}       

\usepackage{moreverb}

\usepackage[dvips,colorlinks,bookmarksopen,bookmarksnumbered,citecolor=red,urlcolor=red]{hyperref}

\newcommand\BibTeX{{\rmfamily B\kern-.05em \textsc{i\kern-.025em b}\kern-.08em
T\kern-.1667em\lower.7ex\hbox{E}\kern-.125emX}}

 \newtheorem{thm}{Theorem}[section]
 
 \newtheorem{lem}[thm]{Lemma}
 \newtheorem{prop}[thm]{Proposition}
 \theoremstyle{definition}
 
 \theoremstyle{remark}
 \newtheorem{rem}[thm]{Remark}
 
 \numberwithin{equation}{section}

\begin{document}

%
\title{A singularly perturbed Dirichlet problem for the Poisson equation in a periodically perforated domain. A functional analytic approach}

\author{Paolo Musolino}

\date{}

\maketitle


\noindent
{\bf Abstract:} Let $\Omega$ be a sufficiently regular bounded open connected subset of $\mathbb{R}^n$ such that $0 \in \Omega$ and that $\mathbb{R}^n \setminus \mathrm{cl}\Omega$ is connected. Then we take $(q_{11},\dots, q_{nn})\in ]0,+\infty[^n$ and $p \in Q\equiv \prod_{j=1}^{n}]0,q_{jj}[$. If $\epsilon$ is a small positive number, then we define the periodically perforated domain $\mathbb{S}[\Omega_{p,\epsilon}]^{-} \equiv \mathbb{R}^n\setminus \cup_{z \in \mathbb{Z}^n}\mathrm{cl}\bigl(p+\epsilon \Omega +\sum_{j=1}^n (q_{jj}z_j)e_j\bigr)$, where $\{e_1,\dots,e_n\}$ is the canonical basis of $\mathbb{R}^n$. For $\epsilon$ small and positive, we introduce a particular Dirichlet problem for the Poisson equation in the set $\mathbb{S}[\Omega_{p,\epsilon}]^{-}$. Namely, we consider a Dirichlet condition on the boundary of the set $p+\epsilon \Omega$, together with a periodicity condition. Then we show real analytic continuation properties of the solution as a function of $\epsilon$, of the Dirichlet datum on $p+\epsilon \partial \Omega$, and of the Poisson datum, around a degenerate triple with $\epsilon=0$.

\vspace{11pt}




\section{Introduction}
In this article, we consider a Dirichlet problem for the Poisson equation in a periodically perforated domain with small holes. We fix once for all  a natural number
\[
n\in {\mathbb{N}}\setminus\{0,1 \}
\]
and
\[
 (q_{11},\dots,q_{nn})\in]0,+\infty[^{n}
\]
and a periodicity cell
\[
Q\equiv\Pi_{j=1}^{n}]0,q_{jj}[\,.
\]
Then we denote by $\mathrm{meas}(Q)$ the measure of the fundamental cell $Q$, and by $\nu_Q$ the outward unit normal to $\partial Q$, where it exists. We denote by $q$ the $n\times n$ diagonal matrix defined by $q\equiv \left(\delta_{i,j}q_{jj}\right)_{i,j=1,\dots,n}$, where $\delta_{i,j}=1$ if $i=j$ and $\delta_{i,j}=0$ if $i \neq j$, for all $i,j\in\{1,\dots,n\}$. Clearly, $q {\mathbb{Z}}^{n}\equiv  \{qz:\,z\in{\mathbb{Z}}^{n}\}$ is the set of vertices of a periodic subdivision of ${\mathbb{R}}^{n}$ corresponding to the fundamental cell $Q$. Let $m\in {\mathbb{N}}\setminus\{0\}$, $\alpha\in]0,1[$. Then we take a point $p \in Q$ and a bounded open connected subset $\Omega$ of ${\mathbb{R}}^{n}$ of class $C^{m,\alpha}$ such that $\Omega^{-}\equiv {\mathbb{R}}^{n}\setminus{\mathrm{cl}}\Omega$ is connected and that $0 \in \Omega$. Here `$\mathrm{cl}$' denotes the closure. If $\epsilon \in \mathbb{R}$, then we set
\[
\Omega_{p,\epsilon}\equiv p +\epsilon \Omega\,.
\]
Then we take $\epsilon_0>0$ such that $\mathrm{cl}\Omega_{p,\epsilon} \subseteq Q$ for $|\epsilon|< \epsilon_0$, and we introduce the periodically perforated domain
\[
{\mathbb{S}} [\Omega_{p,\epsilon}]^{-}\equiv {\mathbb{R}}^{n}\setminus\bigcup_{z \in \mathbb{Z}^n}{\mathrm{cl}}(\Omega_{p,\epsilon}+qz)\,,
\]
for $\epsilon \in ]-\epsilon_0,\epsilon_0[$. Let $\rho>0$. Next we fix a function $g_0$ in the Schauder space $C^{m,\alpha}(\partial \Omega)$ and a function $f_0$ in the Roumieu class $C^0_{q,\omega,\rho}(\mathbb{R}^n)_0$ of real analytic periodic functions with vanishing integral on $Q$ (see \eqref{proum} and \eqref{proum0} in Section \ref{not}.) For each triple $(\epsilon,g,f) \in ]0,\epsilon_0[\times C^{m,\alpha}(\partial \Omega) \times C^0_{q,\omega,\rho}(\mathbb{R}^n)_0$ we consider the Dirichlet problem 
 \begin{equation}\label{bvp:Direps}
 \left \lbrace 
 \begin{array}{ll}
 \Delta u (x)= f(x) & \textrm{$\forall x \in {\mathbb{S}} [\Omega_{p,\epsilon}]^{-}$}\,, \\
u(x+qe_j) =u(x) &  \textrm{$\forall x \in \mathrm{cl} {\mathbb{S}} [\Omega_{p,\epsilon}]^{-}\,, \quad \forall j \in \{1,\dots,n\}$}\,, \\
u(x)=g\bigl((x-p)/\epsilon\bigr) & \textrm{$\forall x \in \partial \Omega_{p,\epsilon}$}\,,
 \end{array}
 \right.
 \end{equation}
where $\{e_1,\dots,e_n\}$ is the canonical basis of $\mathbb{R}^n$. If $(\epsilon,g,f) \in ]0,\epsilon_0[\times C^{m,\alpha}(\partial \Omega)\times  C^0_{q,\omega,\rho}(\mathbb{R}^n)_0$, then problem \eqref{bvp:Direps} has a unique solution in $C^{m,\alpha}(\mathrm{cl} {\mathbb{S}} [\Omega_{p,\epsilon}]^{-})$, and we denote it by $u[\epsilon,g,f](\cdot)$ (cf.~Proposition \ref{prop:Dirsol}.)

Then we pose the following questions.
\begin{enumerate}
\item[(i)] Let $x$ be fixed in $\mathbb{R}^n \setminus (p+q\mathbb{Z}^n)$. What can be said on the map $(\epsilon,g,f) \mapsto u[\epsilon,g,f](x)$ when $(\epsilon,g,f)$ approaches $(0,g_0,f_0)$?
\item[(ii)] What can be said on the map $(\epsilon,g,f) \mapsto\int_{Q\setminus \mathrm{cl} \Omega_{p,\epsilon}}|D_x u[\epsilon,g,f](x)|^2 \, dx$ when $(\epsilon,g,f)$ approaches $(0,g_0,f_0)$?
\item[(iii)] What can be said on the map $(\epsilon,g,f) \mapsto\int_{Q\setminus \mathrm{cl} \Omega_{p,\epsilon}}u[\epsilon,g,f](x) \, dx$ when $(\epsilon,g,f)$ approaches $(0,g_0,f_0)$?
\end{enumerate}

Questions of this type have long been investigated with the methods of Asymptotic Analysis, which aim at computing, for example, an asymptotic expansion of the value of the solution at a fixed point in terms of the parameter $\epsilon$. Here, we mention, \textit{e.g.}, Ammari and Kang \cite[Ch.~5]{AmKa07}, Kozlov, Maz'ya, and Movchan \cite{KoMaMo99}, Maz'ya, Nazarov, and Plamenewskij \cite{MaNaPl00}, Ozawa \cite{Oz83}, Ward and Keller \cite{WaKe93}. We also mention the vast literature of Calculus of Variations and of Homogenization Theory, where the interest is focused on the limiting behaviour as the singular perturbation parameter  degenerates (cf.~\textit{e.g.}, Cioranescu and Murat \cite{CiMu82i, CiMu82ii}.) 

Here instead we wish to characterize the behaviour of $u[\epsilon,g,f](\cdot)$ at $(\epsilon,g,f)=(0,g_0,f_0)$ by a different approach. Thus for example, if we consider a certain function, say $F(\epsilon,g,f)$, relative to the solution such as, for example, one of those considered in questions (i)--(iii) above, we would try to prove that $F(\cdot,\cdot,\cdot)$ can be continued real analytically around $(\epsilon,g,f)=(0,g_0,f_0)$. We observe that our approach does have its advantages. Indeed, if we know that $F(\epsilon,g,f)$ equals, for $\epsilon>0$, a real analytic operator of the variable $(\epsilon, g,f)$ defined on a whole neighborhood of $(0,g_0,f_0)$, then, for example, we could infer the existence of a sequence of real numbers $\{a_j\}_{j=0}^{\infty}$ and of a positive real number $\epsilon' \in ]0,\epsilon_0[$ such that
\[
F(\epsilon,g_0,f_0)=\sum_{j=0}^{\infty}a_j \epsilon^{j} \qquad \forall \epsilon \in ]0,\epsilon'[\, , 
\] 
where the power series in the right hand side converges absolutely on the whole of $]-\epsilon',\epsilon'[$. Such a project has been carried out by Lanza de Cristoforis in several papers for problems in a bounded domain with a small hole (cf.~\textit{e.g.}, Lanza \cite{La02, La04,La08,La09a}.) In the frame of linearized elastostatics, we mention, \textit{e.g.}, Dalla Riva and Lanza \cite{DaLa10b, DaLa10}, and for a boundary value problem for the Stokes system in a singularly perturbed exterior domain we refer to Dalla Riva \cite{Da11}. Instead, for periodic problems, we mention \cite{LaMu11,Mu11}, where Dirichlet and nonlinear Robin problems for the Laplace equation have been considered.

As far as problems in periodically perforated domains are concerned, we mention, for instance, Ammari and Kang \cite[Ch. 8]{AmKa07}.  We also observe that boundary value problems in domains with periodic inclusions can be analysed, at least for the two dimensional case, with the method of functional equations. Here we mention, \textit{e.g.}, Castro and Pesetskaya \cite{CaPe10}, Castro, Pesetskaya, and Rogosin \cite{CaPeRo09}, Drygas and Mityushev \cite{DrMi09}, Mityushev and Adler \cite{MiAd02i}, Rogosin, Dubatovskaya, and Pesetskaya \cite{RoDuPe09}. In connection with doubly periodic problems for composite materials, we mention the monograph of Grigolyuk  and Fil'shtinskij \cite{GrFi92}.\par

 We now briefly outline our strategy. By means of a periodic analog of the Newtonian potential, we can convert problem \eqref{bvp:Direps} into an auxiliary Dirichlet problem for the Laplace equation. Next we analyze the dependence of the solution of the auxiliary problem upon the triple $(\epsilon,g,f)$ by exploiting the results of \cite{Mu11} on the homogeneous Dirichlet problem and of Lanza \cite{La05} on the Newtonian potential in the Roumieu classes, and then we deduce the representation of $u[\epsilon,g,f](\cdot)$ in terms of $\epsilon$, $g$, and $f$, and we prove our main results.

This article is organized as follows. Section \ref{not} is a section of preliminaries. In Section \ref{form}, we formulate the auxiliary Dirichlet problem for the Laplace equation and we show that the solutions of the auxiliary problem depend real analytically on $\epsilon$, $g$, and $f$. In Section \ref{rep}, we show that the results of Section \ref{form} can be exploited to prove Theorem \ref{thm:rep} on the representation of $u[\epsilon,g,f](\cdot)$, Theorems \ref{thm:en} and \ref{thm:int} on the representation of the energy integral  and of the integral of the solution, respectively. At the end of the paper, we have included an Appendix on a slight variant of a result on composition operators of Preciso (cf.~Preciso \cite[Prop.~4.2.16, p.~51]{Pr98}, Preciso \cite[Prop.~1.1, p.~101]{Pr00}), which belongs to a different flow of ideas and which accordingly we prefer to discuss separately.

\section{Notation and preliminaries}\label{not}

We  denote the norm on 
a   normed space ${\mathcal X}$ by $\|\cdot\|_{{\mathcal X}}$. Let 
${\mathcal X}$ and ${\mathcal Y}$ be normed spaces. We endow the  
space ${\mathcal X}\times {\mathcal Y}$ with the norm defined by 
$\|(x,y)\|_{{\mathcal X}\times {\mathcal Y}}\equiv \|x\|_{{\mathcal X}}+
\|y\|_{{\mathcal Y}}$ for all $(x,y)\in  {\mathcal X}\times {\mathcal 
Y}$, while we use the Euclidean norm for ${\mathbb{R}}^{n}$.
 For 
standard definitions of Calculus in normed spaces, we refer to 
Prodi and Ambrosetti~\cite{PrAm73}. The symbol ${\mathbb{N}}$ denotes the 
set of natural numbers including $0$. A 
dot ``$\cdot$'' denotes the inner product in ${\mathbb R}^{n}$. Let $A$ be a 
matrix. Then $A_{ij}$ denotes 
the $(i,j)$--entry of $A$, and the inverse of the 
matrix $A$ is denoted by $A^{-1}$. Let 
${\mathbb{D}}\subseteq {\mathbb {R}}^{n}$. Then $\mathrm{cl}{\mathbb{D}}$ 
denotes the 
closure of ${\mathbb{D}}$ and $\partial{\mathbb{D}}$ denotes the boundary of ${\mathbb{D}}$. For all $R>0$, $ x\in{\mathbb{R}}^{n}$, 
$x_{j}$ denotes the $j$--th coordinate of $x$, 
$| x|$ denotes the Euclidean modulus of $ x$ in
${\mathbb{R}}^{n}$, and ${\mathbb{B}}_{n}( x,R)$ denotes the ball $\{
y\in{\mathbb{R}}^{n}:\, | x- y|<R\}$. 
Let $\Omega$ be an open 
subset of ${\mathbb{R}}^{n}$. The space of $m$ times continuously 
differentiable real-valued functions on $\Omega$ is denoted by 
$C^{m}(\Omega,{\mathbb{R}})$, or more simply by $C^{m}(\Omega)$. 
${\mathcal{D}}(\Omega)$ denotes the space of functions of  $C^{\infty}(\Omega)$
with compact support. The dual ${\mathcal{D}}'(\Omega)$ denotes the space
of distributions in $\Omega$.
Let $r \in \mathbb{N}\setminus \{0\}$. Let $f\in \left(C^{m}(\Omega)\right)^{r}$. The 
$s$--th component of $f$ is denoted $f_{s}$, and $Df$ denotes the Jacobian matrix
$\left(\frac{\partial f_s}{\partial
x_l}\right)_{  \substack{ s=1,\dots ,r,    \\  l=1,\dots ,n}       }$. Let  $\eta\equiv
(\eta_{1},\dots ,\eta_{n})\in{\mathbb{N}}^{n}$, $|\eta |\equiv
\eta_{1}+\dots +\eta_{n}  $. Then $D^{\eta} f$ denotes
$\frac{\partial^{|\eta|}f}{\partial
x_{1}^{\eta_{1}}\dots\partial x_{n}^{\eta_{n}}}$.    The
subspace of $C^{m}(\Omega )$ of those functions $f$ whose derivatives $D^{\eta }f$ of
order $|\eta |\leq m$ can be extended with continuity to 
$\mathrm{cl}\Omega$  is  denoted $C^{m}(
\mathrm{cl}\Omega )$. 
The
subspace of $C^{m}(\mathrm{cl}\Omega ) $  whose
functions have $m$--th order derivatives that are
H\"{o}lder continuous  with exponent  $\alpha\in
]0,1]$ is denoted $C^{m,\alpha} (\mathrm{cl}\Omega )$  
(cf.~\textit{e.g.},~Gilbarg and 
Trudinger~\cite{GiTr83}.) The subspace of $C^{m}(\mathrm{cl}\Omega ) $ of those functions $f$ such that $f_{|{\mathrm{cl}}(\Omega\cap{\mathbb{B}}_{n}(0,R))}\in
C^{m,\alpha}({\mathrm{cl}}(\Omega\cap{\mathbb{B}}_{n}(0,R)))$ for all $R\in]0,+\infty[$ is denoted $C^{m,\alpha}_{{\mathrm{loc}}}(\mathrm{cl}\Omega ) $.  Let 
${\mathbb{D}}\subseteq {\mathbb{R}}^{r}$. Then $C^{m
,\alpha }(\mathrm{cl}\Omega ,{\mathbb{D}})$ denotes
$\left\{f\in \left( C^{m,\alpha} (\mathrm{cl}\Omega )\right)^{r}:\ f(
\mathrm{cl}\Omega )\subseteq {\mathbb{D}}\right\}$. \par
Now let $\Omega $ be a bounded
open subset of  ${\mathbb{R}}^{n}$. Then $C^{m}(\mathrm{cl}\Omega )$ 
and $C^{m,\alpha }({\mathrm{cl}}
\Omega )$ are endowed with their usual norm and are well known to be 
Banach spaces  (cf.~\textit{e.g.}, Troianiello~\cite[\S 1.2.1]{Tr87}.) 
We say that a bounded open subset $\Omega$ of ${\mathbb{R}}^{n}$ is of class 
$C^{m}$ or of class $C^{m,\alpha}$, if its closure is a 
manifold with boundary imbedded in 
${\mathbb{R}}^{n}$ of class $C^{m}$ or $C^{m,\alpha}$, respectively
 (cf.~\textit{e.g.}, Gilbarg and Trudinger~\cite[\S 6.2]{GiTr83}.) 
We denote by 
$
\nu_{\Omega}
$
the outward unit normal to $\partial\Omega$.  For standard properties of functions 
in Schauder spaces, we refer the reader to Gilbarg and 
Trudinger~\cite{GiTr83} and to Troianiello~\cite{Tr87}
(see also Lanza \cite[\S 2, Lem.~3.1, 4.26, Thm.~4.28]{La91},  Lanza and Rossi \cite[\S 2]{LaRo04}.) \par
 
If $\mathbb{M}$ is a manifold  imbedded in 
${\mathbb{R}}^{n}$ of class $C^{m,\alpha}$, with $m\geq 1$, 
$\alpha\in ]0,1]$, one can define the Schauder spaces also on $\mathbb{M}$ by 
exploiting the local parametrizations. In particular, one can consider 
the spaces $C^{k,\alpha}(\partial\Omega)$ on $\partial\Omega$ for 
$0\leq k\leq m$ with $\Omega$ a bounded open set of class $C^{m,\alpha}$,
and the trace operator of $C^{k,\alpha}({\mathrm{cl}}\Omega)$ to
$C^{k,\alpha}(\partial\Omega)$ is linear and continuous. We denote by $d\sigma$ the area element of a manifold imbedded in ${\mathbb{R}}^{n}$. \par

We note that 
throughout the paper ``analytic'' means ``real analytic''. For the 
definition and properties of real analytic operators, we refer to Prodi and 
Ambrosetti~\cite[p.~89]{PrAm73}. In particular, we mention that the 
pointwise product in Schauder spaces is bilinear and continuous, and 
thus analytic (cf.~\textit{e.g.}, Lanza and Rossi \cite[p.~141]{LaRo04}.)\par

For all bounded open subsets $\Omega$ of $\mathbb{R}^n$ and $\rho>0$, we set
\[
C_{\omega,\rho}^0(\mathrm{cl}\Omega) \equiv \Bigl\{u \in C^\infty(\mathrm{cl}\Omega)\colon \sup_{\beta \in \mathbb{N}^n}\frac{\rho^{|\beta|}}{|\beta|!}\|D^\beta u\|_{C^0(\mathrm{cl}\Omega)}<+\infty \Bigr\}\,,
\]
and 
\[
\|u\|_{C_{\omega,\rho}^0(\mathrm{cl}\Omega)}\equiv  \sup_{\beta \in \mathbb{N}^n}\frac{\rho^{|\beta|}}{|\beta|!}\|D^\beta u\|_{C^0(\mathrm{cl}\Omega)} \qquad \forall u \in C_{\omega,\rho}^0(\mathrm{cl}\Omega)\,,
\]
where $|\beta|\equiv \beta_1+\dots+\beta_n$, for all $\beta\equiv(\beta_1,\dots,\beta_n)\in \mathbb{N}^n$. As is well known, the Roumieu class $\bigl(C_{\omega,\rho}^0(\mathrm{cl}\Omega),\|\cdot\|_{C_{\omega,\rho}^0(\mathrm{cl}\Omega)}\bigr)$ is a Banach space. \par

A straightforward computation based on the inequality $\binom{\beta}{\alpha}\leq \binom{|\beta|}{|\alpha|}$, which holds for $\alpha$, $\beta \in \mathbb{N}^n$, $0\leq \alpha\leq \beta$, shows that the pointwise product is bilinear and continuous from $\bigl(C^0_{\omega,\rho}(\mathrm{cl}\Omega)\bigr)^2$ to $C^0_{\omega,\rho'}(\mathrm{cl}\Omega)$ for all $\rho' \in ]0,\rho[$.

Let $\mathbb{D}\subseteq \mathbb{R}^n$ be such that $qz+\mathbb{D}\subseteq \mathbb{D}$ for all $z \in \mathbb{Z}^n$. We say that a function $f$ from $\mathbb{D}$ to $\mathbb{R}$ is $q$--periodic if $f(x+qe_j)=f(x)$ for all $x \in \mathbb{D}$ and for all $j\in \{1,\dots,n\}$. If $k \in \mathbb{N}$, then we set
\[
C^{k}_{q}(\mathbb{R}^n)\equiv
\{
u\in C^{k}(\mathbb{R}^n):\,
u\ {\mathrm{is}}\ q-{\mathrm{periodic}}
\}\,.
\]
We also set
\[
C^{\infty}_{q}(\mathbb{R}^n)\equiv
\{
u\in C^{\infty}(\mathbb{R}^n):\,
u\ {\mathrm{is}}\ q-{\mathrm{periodic}}
\}\,.
\]
Similarly, if $\rho>0$, we set
\begin{equation}
\label{proum}
C_{q,\omega,\rho}^0(\mathbb{R}^n) \equiv \Bigl\{u \in C^\infty_q(\mathbb{R}^n)\colon \sup_{\beta \in \mathbb{N}^n}\frac{\rho^{|\beta|}}{|\beta|!}\|D^\beta u\|_{C^0(\mathrm{cl}Q)}<+\infty \Bigr\}\,,
\end{equation}
and 
\[
\|u\|_{C_{q,\omega,\rho}^0(\mathbb{R}^n)}\equiv  \sup_{\beta \in \mathbb{N}^n}\frac{\rho^{|\beta|}}{|\beta|!}\|D^\beta u\|_{C^0(\mathrm{cl}Q)} \qquad \forall u \in C_{q,\omega,\rho}^0(\mathbb{R}^n)\,.
\]
As can be easily seen, the periodic Roumieu class $\bigl(C_{q,\omega,\rho}^0(\mathbb{R}^n),\|\cdot\|_{C_{q,\omega,\rho}^0(\mathbb{R}^n)}\bigr)$ is a Banach space. Obviously, the restriction operator from $C_{q,\omega,\rho}^0(\mathbb{R}^n)$ to $C_{\omega,\rho}^0(\mathrm{cl}\mathbb{D})$ is linear and continuous for all bounded open subsets $\mathbb{D}$ of $\mathbb{R}^n$ and $\rho>0$. Similarly, if $\rho>0$, if $\mathbb{D}$ is a bounded open subset of $\mathbb{R}^n$ such that $\mathrm{cl}Q \subseteq \mathbb{D}$, and if $f \in C^0_q(\mathbb{R}^n)$ is such that $f_{|\mathrm{cl}\mathbb{D}} \in C^0_{\omega,\rho}(\mathrm{cl}\mathbb{D})$, then clearly $f \in C^0_{q,\omega,\rho}(\mathbb{R}^n)$. 

If $\Omega$ is an open subset of ${\mathbb{R}}^{n}$, $k\in {\mathbb{N}}$, $\beta\in]0,1]$, we set
\[
C^{k,\beta}_{b}({\mathrm{cl}}\Omega)\equiv
\{
u\in C^{k,\beta}({\mathrm{cl}}\Omega):\,
D^{\gamma}u\ {\mathrm{is\ bounded}}\ \forall\gamma\in {\mathbb{N}}^{n}\
{\mathrm{such\ that}}\ |\gamma|\leq k
\}\,,
\]
and we endow $C^{k,\beta}_{b}({\mathrm{cl}}\Omega)$ with its usual  norm
\[
\|u\|_{ C^{k,\beta}_{b}({\mathrm{cl}}\Omega) }\equiv
\sum_{|\gamma|\leq k}\sup_{x\in {\mathrm{cl}}\Omega }|D^{\gamma}u(x)|
+\sum_{|\gamma| = k}|D^{\gamma}u: {\mathrm{cl}}\Omega |_{\beta}
\qquad\forall u\in C^{k,\beta}_{b}({\mathrm{cl}}\Omega)\,,
\]
where $|D^{\gamma}u: {\mathrm{cl}}\Omega |_{\beta}$ denotes the $\beta$-H\"{o}lder constant of $D^{\gamma}u$.

Next we turn to periodic domains. If $\mathbb{I}$ is an arbitrary subset of ${\mathbb{R}}^{n}$  such that
 ${\mathrm{cl}}\mathbb{I}\subseteq Q$, then we set
\[
{\mathbb{S}} [\mathbb{I}]\equiv 
\bigcup_{z\in{\mathbb{Z}}^{n} }(qz+\mathbb{I})=q{\mathbb{Z}}^{n}+\mathbb{I}\,,\qquad{\mathbb{S}} [\mathbb{I}]^{-}\equiv {\mathbb{R}}^{n}\setminus{\mathrm{cl}}{\mathbb{S}} [\mathbb{I}]\,.
\]
We note that if $\mathbb{R}^n\setminus \mathrm{cl}\mathbb{I}$ is connected, then $\mathbb{S}[\mathbb{I}]^{-}$ is connected. 

If $\mathbb{I}$ is an open subset of ${\mathbb{R}}^{n}$  such that ${\mathrm{cl}}\mathbb{I}\subseteq Q$ and if $k\in {\mathbb{N}}$, $\beta\in]0,1]$, then we set
\[
C^{k,\beta}_{q}({\mathrm{cl}}{\mathbb{S}}[\mathbb{I}] )
\equiv\left\{
u\in C^{k,\beta}_b({\mathrm{cl}}{\mathbb{S}}[\mathbb{I}] ):\,
u\ {\mathrm{is}}\ q-{\mathrm{periodic}}
\right\}\,,
\]
which we regard as a Banach subspace of $C^{k,\beta}_{b}({\mathrm{cl}}{\mathbb{S}}[\mathbb{I}] )$, and
\[
C^{k,\beta}_{q}({\mathrm{cl}}{\mathbb{S}}[\mathbb{I}]^{-} )
\equiv\left\{
u\in C^{k,\beta}_b({\mathrm{cl}}{\mathbb{S}}[\mathbb{I}]^{-} ):\,
u\ {\mathrm{is}}\ q-{\mathrm{periodic}}
\right\}\,,
\]
which we regard as a Banach subspace of $C^{k,\beta}_{b}({\mathrm{cl}}{\mathbb{S}}[\mathbb{I}]^{-})$.  

The Laplace operator is well known to have  a $\{0\}$-analog of a $q$-periodic fundamental solution, \textit{i.e.}, a $q$-periodic tempered distribution $S_{q,n}$ such that
\[
\Delta S_{q,n}=\sum_{z\in {\mathbb{Z}}^{n}}\delta_{qz}-\frac{1}{\mathrm{meas}(Q)}\,,
\]
where $\delta_{qz}$ denotes the Dirac measure with mass in $qz$ (cf.~\textit{e.g.}, \cite[\S 3]{LaMu10a}.)  
As is well known, $S_{q,n}$ is determined  up to an additive constant, and we can take
\[
S_{q,n}(x)=-\sum_{ z\in {\mathbb{Z}}^{n}\setminus\{0\} }
\frac{1}{     \mathrm{meas}(Q)4\pi^{2}|q^{-1}z|^{2}   }e^{2\pi i (q^{-1}z)\cdot x}
\qquad{\mathrm{in}}\ {\mathcal{D}}'({\mathbb{R}}^{n})
\]
(cf.~\textit{e.g.}, Ammari and Kang \cite[p.~53]{AmKa07}, \cite[\S 3]{LaMu10a}.) 
Clearly $S_{q,n}$ is even.  Moreover, $S_{q,n}$ is real analytic in ${\mathbb{R}}^{n}\setminus q{\mathbb{Z}}^{n}$ and is locally integrable in ${\mathbb{R}}^{n}$.\par

Let
$S_{n}$ be the function from ${\mathbb{R}}^{n}\setminus\{0\}$ to ${\mathbb{R}}$ defined by
\[
S_{n}(x)\equiv
\left\{
\begin{array}{lll}
\frac{1}{s_{n}}\log |x| \qquad &   \forall x\in 
{\mathbb{R}}^{n}\setminus\{0\},\quad & {\mathrm{if}}\ n=2\,,
\\
\frac{1}{(2-n)s_{n}}|x|^{2-n}\qquad &   \forall x\in 
{\mathbb{R}}^{n}\setminus\{0\},\quad & {\mathrm{if}}\ n>2\,,
\end{array}
\right.
\]
where $s_{n}$ denotes the $(n-1)$ dimensional measure of 
$\partial{\mathbb{B}}_{n}$. $S_{n}$ is well-known to be the 
fundamental solution of the Laplace operator. \par

Then the function $S_{q,n}-S_{n}$ is analytic in $({\mathbb{R}}^{n}\setminus q{\mathbb{Z}}^{n})\cup\{0\}$ (cf.~\textit{e.g.}, Ammari and Kang \cite[Lemma~2.39, p.~54]{AmKa07}, \cite[\S 3]{LaMu10a}.) Then we find convenient to set
\[
R_n\equiv S_{q,n}-S_{n}\qquad{\mathrm{in}}\ 
({\mathbb{R}}^{n}\setminus q{\mathbb{Z}}^{n})\cup\{0\}\,.
\]

We now introduce the periodic simple layer potential. Let  $\alpha\in]0,1[$, $m\in {\mathbb{N}}\setminus\{0\}$. Let $\mathbb{I}$ be a bounded open subset of ${\mathbb{R}}^{n}$ of class $C^{m,\alpha}$ such that ${\mathrm{cl}}\mathbb{I}\subseteq Q$. If $\mu\in C^{0,\alpha}(\partial\mathbb{I})$, we set
\[
v_{q}[\partial\mathbb{I},\mu](x)\equiv
\int_{\partial\mathbb{I}}S_{q,n}(x-y)\mu(y)\,d\sigma_{y}
\qquad\forall x\in {\mathbb{R}}^{n}\,.
\]
As is well known, if $\mu\in C^{m-1,\alpha}(\partial\mathbb{I})$, then the function 
$v^{+}_{q}[\partial\mathbb{I},\mu]\equiv v_{q}[\partial\mathbb{I},\mu]_{|{\mathrm{cl}}{\mathbb{S}}[\mathbb{I}]}$ belongs to $C^{m,\alpha}_{q}({\mathrm{cl}}{\mathbb{S}}[\mathbb{I}])$, and the function 
$v^{-}_{q}[\partial\mathbb{I},\mu]\equiv v_{q}[\partial\mathbb{I},\mu]_{|{\mathrm{cl}}{\mathbb{S}}[\mathbb{I}]^{-}}$ belongs to $C^{m,\alpha}_{q}
({\mathrm{cl}}{\mathbb{S}}[\mathbb{I}]^{-})$. Similarly, we introduce the periodic double layer potential. If $\mu\in C^{0,\alpha}(\partial\mathbb{I})$, we set
\[
w_{q}[\partial\mathbb{I},\mu](x)\equiv
-\int_{\partial\mathbb{I}}(DS_{q,n}(x-y) )\nu_{\mathbb{I}}(y)\mu(y)\,d\sigma_{y}
\qquad\forall x\in {\mathbb{R}}^{n}\,.
\]
If $\mu\in C^{m,\alpha}(\partial\mathbb{I})$, then 
the restriction $w_{q}[\partial\mathbb{I},\mu]_{|{\mathbb{S}}[\mathbb{I}]}$ can be extended uniquely to an element $w^{+}_{ q }[\partial\mathbb{I},\mu]$ of $C_{q}^{m,\alpha}({\mathrm{cl}}{\mathbb{S}}[\mathbb{I}])$, and the  restriction $w_{q}[\partial\mathbb{I},\mu]_{|{\mathbb{S}}[\mathbb{I}]^{-}}$ can be extended uniquely to an element $w_{ q }^{-}[\partial\mathbb{I},\mu]$ of $C^{m,\alpha}_{q}({\mathrm{cl}}{\mathbb{S}}[\mathbb{I}]^{-})$, and we have
$w^{\pm }_{q}[\partial\mathbb{I},\mu]=\pm\frac{1}{2}\mu+
w_{q}[\partial\mathbb{I},\mu]$ on $\partial \mathbb{I}$. Moreover, if $\mu\in C^{0,\alpha}(\partial\mathbb{I})$, then we have
\begin{equation}\label{perpot}
w_{q}[\partial\mathbb{I},\mu](x)=-
\sum_{j=1}^{n}\frac{\partial}{\partial x_{j}}
\int_{\partial\mathbb{I}}  
S_{q,n}(x-y)(\nu_{\mathbb{I}}(y))_{j}\mu(y)\,d\sigma_{y}\, ,
\end{equation}
for all  $x\in{\mathbb{R}}^{n}\setminus\partial{\mathbb{S}}[\mathbb{I}]$ (cf.~\textit{e.g.}, \cite[\S 3]{LaMu10a}.)

As we have done for the periodic layer potentials, we now introduce a periodic analog of the Newtonian potential. If $f\in C^{0}_q(\mathbb{R}^n)$, then we set
\[
P_{q}[f](x)\equiv
\int_{Q}S_{q,n}(x-y)f(y)\,dy
\qquad\forall x\in {\mathbb{R}}^{n}\,.
\]
Clearly, $P_q[f]$ is a $q$--periodic function on $\mathbb{R}^n$. 

In the following Theorem, we collect some elementary properties of the periodic Newtonian potential. A proof can be effected by splitting $S_{q,n}$ into the sum of $S_n$ and $R_n$, and by exploiting the results of Lanza \cite{La05} on the (classical) Newtonian potential in the Roumieu classes and standard properties of integral  operators with real analytic kernels and with no singularity (cf.~\textit{e.g.}, \cite[\S 3]{LaMu10b} and \cite{DaLaMu11}.)
\begin{thm}\label{thm:newperpot}
The following statements hold.
\begin{enumerate}
\item[(i)] Let $\beta \in ]0,1]$. Let $f \in C^{0,\beta}_{q}(\mathbb{R}^n)$. Then $P_{q}[f] \in C^{2}_{q}(\mathbb{R}^n)$ and
\[
\Delta P_{q}[f](x)=f(x)-\frac{1}{\mathrm{meas}(Q)}\int_{Q}f(y)\,dy \qquad \forall x \in \mathbb{R}^n\,.
\]
\item[(ii)] Let $\rho>0$. Then there exists $\rho' \in ]0,\rho]$ such that $P_q[f] \in C^0_{q,\omega,\rho'}(\mathbb{R}^n)$ for all $f \in C^0_{q,\omega,\rho}(\mathbb{R}^n)$ and such that $P_q[\cdot]$ is linear and continuous from $C^0_{q,\omega,\rho}(\mathbb{R}^n)$ to $C^0_{q,\omega,\rho'}(\mathbb{R}^n)$.
\end{enumerate}
\end{thm}

Let $m\in {\mathbb{N}}\setminus\{0\}$, $\alpha\in]0,1[$.  If $\Omega$ is a bounded open subset of ${\mathbb{R}}^{n}$ of class $C^{m,\alpha}$, we find convenient to set
\[
C^{m,\alpha}(\partial \Omega)_0\equiv \left\{\phi \in C^{m,\alpha}(\partial \Omega)\colon \int_{\partial \Omega}\phi \, d\sigma=0\right\}\,.
\]
If  $\rho>0$, we also set
\begin{equation}\label{proum0}
C^{0}_{q,\omega,\rho}(\mathbb{R}^n)_0\equiv \left\{f \in C^{0}_{q,\omega,\rho}(\mathbb{R}^n)\colon \int_{Q}f \, dx=0\right\}\,.
\end{equation}
As the following Proposition shows, a periodic Dirichlet boundary value problem for the Poisson equation in the perforated domain $\mathbb{S}[\mathbb{I}]^{-}$ has a unique solution in $C^{m,\alpha}_q(\mathrm{cl}\mathbb{S}[\mathbb{I}]^{-})$, which can be represented as the sum of a periodic double layer potential, of a costant, and of a periodic Newtonian potential.
\begin{prop}\label{prop:Dirsol}
Let $m\in {\mathbb{N}}\setminus\{0\}$, $\alpha\in]0,1[$. Let $\rho>0$. Let $\mathbb{I}$ be a bounded connected open subset of ${\mathbb{R}}^{n}$ of class $C^{m,\alpha}$ such that 
${\mathbb{R}}^{n}\setminus{\mathrm{cl}}\mathbb{I}$ is connected and that ${\mathrm{cl}}\mathbb{I}\subseteq Q$. Let $\Gamma \in C^{m,\alpha}(\partial \mathbb{I})$. Let $f \in C^0_{q,\omega,\rho}(\mathbb{R}^n)_0$. Then the following boundary value problem
\[
 \left \lbrace 
 \begin{array}{ll}
 \Delta u (x)= f(x) & \textrm{$\forall x \in {\mathbb{S}} [\mathbb{I}]^{-}$}\,, \\
u\ {\mathrm{is}}\  $q$-{\mathrm{periodic\ in}}\  \mathrm{cl}{\mathbb{S}}[\mathbb{I}]^{-}\,,&\\
u(x)=\Gamma(x)& \textrm{$\forall x \in \partial \mathbb{I}$}\,,
 \end{array}
 \right.
\]
 has a unique solution $u \in C^{m,\alpha}_q(\mathrm{cl} \mathbb{S}[\mathbb{I}]^{-})$. Moreover,
\[
u(x)=w_{q}^{-}[\partial \mathbb{I},\mu](x)+\xi +P_q[f](x)\qquad \forall x \in \mathrm{cl}\mathbb{S}[\mathbb{I}]^{-}\,,
\]
where $(\mu,\xi)$ is the unique solution in $C^{m,\alpha}(\partial \mathbb{I})_0\times \mathbb{R}$ of the following integral equation
\[
-\frac{1}{2}\mu(x)+w_{q}[\partial \mathbb{I},\mu](x)+\xi=\Gamma(x)-P_q[f](x) \qquad \forall x \in \partial \mathbb{I}\, .
\]
\end{prop}
 \begin{proof} A proof can be effected by considering the difference $u- P_q[f]$ and by solving the corresponding homogeneous problem (cf.~\cite[\S 2]{Mu11} and Theorem \ref{thm:newperpot}.) 
\end{proof}

\section{Formulation and analysis of an auxiliary problem}\label{form}

We shall consider the following assumptions for some $\alpha \in ]0,1[$ and for some natural $m \geq 1$.
\begin{equation}\label{ass} 
\begin{split}
&\text{Let $\Omega$ be a bounded connected open subset of ${\mathbb{R}}^{n}$ of class $C^{m,\alpha}$ }\\ 
&\text{such that ${\mathbb{R}}^{n}\setminus{\mathrm{cl}}\Omega$ is connected and that $0 \in\Omega$}.\\
&
\text{Let $p \in Q$}. 
\end{split}
\end{equation}
If $\epsilon \in \mathbb{R}$ and \eqref{ass} holds, we set
\[
\Omega_{p,\epsilon} \equiv p+\epsilon\Omega \,.
\]
Now let $\epsilon_0$ be such that
\begin{equation}\label{eps0}
\epsilon_0>0 \quad \mathrm{and} \quad \mathrm{cl}\Omega_{p,\epsilon} \subseteq Q\quad \forall \epsilon \in ]-\epsilon_0,\epsilon_0[\,.
\end{equation}
A simple topological argument shows that if \eqref{ass} holds, then $\mathbb{S}[\Omega_{p,\epsilon}]^{-}$ is connected, for all $\epsilon \in ]-\epsilon_0,\epsilon_0[$. We also note that
\[
\nu_{\Omega_{p,\epsilon}}(p+\epsilon t)=\mathrm{sgn}(\epsilon)\nu_{\Omega}(t) \qquad \forall t \in \partial \Omega\, ,
\]
for all $\epsilon \in ]-\epsilon_0,\epsilon_0[\setminus \{0\}$, where $\mathrm{sgn}(\epsilon)=1$ if $\epsilon>0$, $\mathrm{sgn}(\epsilon)=-1$ if $\epsilon<0$. Let $\rho>0$. Then we shall consider also the following assumptions.
\begin{equation}\label{assg}
\text{Let $g_0 \in C^{m,\alpha}(\partial \Omega)$.}
\end{equation}
\begin{equation}\label{assf}
\text{Let $f_0 \in C^{0}_{q,\omega,\rho}(\mathbb{R}^n)_0$.}
\end{equation}
If $(\epsilon,g,f) \in ]0,\epsilon_0[\times C^{m,\alpha}(\partial \Omega)\times C^0_{q,\omega,\rho}(\mathbb{R}^n)_0$, then we denote by $u[\epsilon,g,f]$ the unique solution in $C^{m,\alpha}_q(\mathrm{cl}\mathbb{S}[\Omega_{p,\epsilon}]^-)$ of problem \eqref{bvp:Direps}, and by $u_\#[\epsilon,g,f]$ the unique solution in $C^{m,\alpha}_q(\mathrm{cl}\mathbb{S}[\Omega_{p,\epsilon}]^-)$ of the following auxiliary boundary value problem
 \begin{equation}\label{bvp:auxDireps}
 \left \lbrace 
 \begin{array}{ll}
 \Delta u (x)= 0 & \textrm{$\forall x \in {\mathbb{S}} [\Omega_{p,\epsilon}]^{-}$}\,, \\
u\ {\mathrm{is}}\  $q$-{\mathrm{periodic\ in}}\  \mathrm{cl}{\mathbb{S}}[\Omega_{p,\epsilon}]^{-}\,,&\\
u(x)=g\bigl((x-p)/\epsilon\bigr) -P_q[f](x)& \textrm{$\forall x \in \partial \Omega_{p,\epsilon}$}\,.
 \end{array}
 \right.
 \end{equation}
 Clearly,
\begin{equation}\label{eq:sum}
u[\epsilon,g,f]= u_\#[\epsilon,g,f]+P_q[f] \qquad \text{on $\mathrm{cl}\mathbb{S}[\Omega_{p,\epsilon}]^-$}\,.
\end{equation}
Let $f \in C^0_{q,\omega,\rho}(\mathbb{R}^n)$, $\epsilon \in ]0,\epsilon_0[$. We note that
\[
P_q[f](x)=P_q[f]\Bigl(p+\epsilon\bigr((x-p)/\epsilon\bigr)\Bigr) \qquad \forall x \in \partial \Omega_{p,\epsilon}\, .
\] 
Accordingly, the Dirichlet condition in problem \eqref{bvp:auxDireps} can be rewritten as
\[
u(x)=\Bigl(g-P_q[f]\circ(p+\epsilon \mathrm{id}_{\partial \Omega})\Bigr)\bigl((x-p)/\epsilon\bigr)  \qquad \forall x \in \partial \Omega_{p,\epsilon}\,,
\]
where $\mathrm{id}_{\partial \Omega}$ denotes the identity map in $\partial \Omega$. As a consequence, in order to study the dependence of $u_\#[\epsilon,g,f]$ upon $(\epsilon,g,f)$ around $(0,g_0,f_0)$, we can exploit the results of \cite{Mu11}, concerning the dependence of the solution of the Dirichlet problem for the Laplace equation upon $\epsilon$ and the Dirichlet datum. In order to do so, we need to study the regularity of the map from $]-\epsilon_0,\epsilon_0[\times C^0_{q,\omega,\rho}(\mathbb{R}^n)$ to $C^{m,\alpha}(\partial \Omega)$ which takes $(\epsilon,f)$ to the function $P_q[f]\circ (p+\epsilon \mathrm{id}_{\partial \Omega})$.
\begin{lem}\label{lem:Peps}
Let $m\in {\mathbb{N}}\setminus\{0\}$, $\alpha\in]0,1[$. Let $\rho>0$. Let \eqref{ass}--\eqref{assf} hold. Let $\mathrm{id}_{\partial \Omega}$, $\mathrm{id}_{\mathrm{cl\Omega}}$ denote the identity map in $\partial \Omega$ and in $\mathrm{cl}\Omega$, respectively. Then the following statements hold.
\begin{enumerate}
\item[(i)] The map from $]-\epsilon_0,\epsilon_0[\times C^{0}_{q,\omega,\rho}(\mathbb{R}^n)$ to $C^{m,\alpha}(\mathrm{cl} \Omega)$ which takes $(\epsilon,f)$ to $P_q[f]\circ(p+\epsilon \mathrm{id}_{\mathrm{cl} \Omega})$ is real analytic.
\item[(ii)] The map from $]-\epsilon_0,\epsilon_0[\times C^{0}_{q,\omega,\rho}(\mathbb{R}^n)$ to $C^{m,\alpha}(\partial \Omega)$ which takes $(\epsilon,f)$ to $P_q[f]\circ(p+\epsilon \mathrm{id}_{\partial \Omega})$ is real analytic.
\end{enumerate}
\end{lem}
 \begin{proof} We first prove statement (i). By Theorem \ref{thm:newperpot} (ii), there exists $\rho'\in ]0,\rho]$ such that the linear map from $C^0_{q,\omega,\rho}(\mathbb{R}^n)$ to $C^0_{\omega,\rho'}(\mathrm{cl}Q)$ which takes $f$ to $P_q[f]_{|\mathrm{cl}Q}$ is continuous. As a consequence, by Proposition \ref{prop:Pr} of the Appendix, we immediately deduce that the map from $]-\epsilon_0,\epsilon_0[\times C^{0}_{q,\omega,\rho}(\mathbb{R}^n)$ to $C^{m,\alpha}(\mathrm{cl} \Omega)$ which takes $(\epsilon,f)$ to $P_q[f]\circ(p+\epsilon \mathrm{id}_{\mathrm{cl} \Omega})$ is real analytic. By the continuity of the trace operator from $C^{m,\alpha}(\mathrm{cl}\Omega)$ to $C^{m,\alpha}(\partial \Omega)$ and statement (i), we deduce the validity of (ii).
\end{proof}

Then we have the following Lemma  (cf.~\cite[\S 3]{Mu11}.)

\begin{lem}\label{lem:limeq}
Let $m\in {\mathbb{N}}\setminus\{0\}$, $\alpha\in]0,1[$. Let $\rho>0$. Let \eqref{ass}--\eqref{assf} hold. Let $\tau_0$ be the unique solution in $C^{m-1,\alpha}(\partial \Omega)$ of the following problem
\[  
\left \lbrace 
 \begin{array}{ll}
-\frac{1}{2}\tau(t)+\int_{\partial\Omega}(DS_{n}(t-s) )\nu_{\Omega}(t)\tau(s)\,d\sigma_{s}=0 \qquad \forall t \in \partial \Omega\, ,\\
\int_{\partial \Omega}\tau \, d\sigma=1\, . & 
 \end{array}
 \right.
\]
Then equation
\[
-\frac{1}{2}\theta(t)-\int_{\partial\Omega}(DS_{n}(t-s) )\nu_{\Omega}(s)\theta(s)\,d\sigma_{s}+\xi=g_0(t)-P_q[f_0](p)\qquad \forall t \in \partial \Omega\, ,
\]
which we call the \emph{limiting equation}, has a unique solution in $C^{m,\alpha}(\partial \Omega)_0\times \mathbb{R}$, which we denote by $(\tilde{\theta},\tilde{\xi})$. Moreover, 
\[
\tilde{\xi}=\int_{\partial \Omega}g_0\tau_0\, d\sigma-P_q[f_0](p)\,,
\]
and the function $\tilde{u} \in C^{m,\alpha}_{\mathrm{loc}}(\mathbb{R}^n \setminus \mathrm{cl}\Omega)$, defined by
\[
\tilde{u}(t)\equiv -\int_{\partial \Omega}(DS_n(t-s))\nu_{\Omega}(s)\tilde{\theta}(s)d\sigma_s \qquad \forall t \in \mathbb{R}^n \setminus \mathrm{cl}\Omega\, ,
\]
has a unique continuous extension to $\mathbb{R}^{n}\setminus \Omega$, which we still denote by $\tilde{u}$, and such an extension is the unique solution in $C^{m,\alpha}_{\mathrm{loc}}(\mathbb{R}^n\setminus \Omega)$ of the following problem
\[
 \left \lbrace 
 \begin{array}{ll}
 \Delta u (t)= 0 & \textrm{$\forall t \in \mathbb{R}^n \setminus \mathrm{cl}\Omega$}\,, \\
u(t)=g_0(t)-\int_{\partial \Omega}g_0\tau_0\, d\sigma&  \textrm{$\forall t \in \partial \Omega$}\,, \\
\lim_{t\to \infty}u(t)=0\,.&
 \end{array}
 \right.
\]
\end{lem}

In \cite{Mu11}, we have shown that the solutions of a periodic Dirichlet problem for the Laplace equation in $\mathbb{S}[\Omega_{p,\epsilon}]^-$ depend analytically upon $\epsilon$ and upon (a rescaling of) the Dirichlet datum. By Lemma \ref{lem:Peps} (ii), we know that (a rescaling of) the Dirichlet datum of the auxiliary problem \eqref{bvp:auxDireps} depends analytically upon $(\epsilon,g,f)$. Then we deduce that the solution of problem \eqref{bvp:auxDireps} depends analytically on
$(\epsilon, g,f)$, and we have the following.

\begin{prop}\label{prop:Lambda}
Let $m\in {\mathbb{N}}\setminus\{0\}$, $\alpha\in]0,1[$. Let $\rho>0$. Let \eqref{ass}--\eqref{assf} hold. Then there exist $\epsilon_1 \in ]0,\epsilon_0]$, an open neighborhood $\mathcal{U}$ of $g_0$ in $C^{m,\alpha}(\partial \Omega)$, an open neighborhood $\mathcal{V}$ of $f_0$ in $C^{0}_{q,\omega,\rho}(\mathbb{R}^n)_0$, and a real analytic map $(\Theta[\cdot,\cdot,\cdot],\Xi[\cdot,\cdot,\cdot])$ from $]-\epsilon_1,\epsilon_1[\times \mathcal{U}\times \mathcal{V}$ to $C^{m,\alpha}(\partial \Omega)_0\times \mathbb{R}$ such that
\[
u_\#[\epsilon,g,f]=w_q^-[\partial \Omega_{p,\epsilon},\Theta[\epsilon,g,f]((\cdot-p)/\epsilon)]+\Xi[\epsilon,g,f] \qquad \text{on $\mathrm{cl}\mathbb{S}[\Omega_{p,\epsilon}]^-$}\,,
\]
for all $(\epsilon,g,f)\in ]0,\epsilon_1[\times \mathcal{U}\times \mathcal{V}$. Moreover, $(\Theta[0,g_0,f_0],\Xi[0,g_0,f_0])=(\tilde{\theta},\tilde{\xi})$, where $\tilde{\theta}$, $\tilde{\xi}$ are as in Lemma \ref{lem:limeq}.
\end{prop}

Then we have the following representation Theorem for $u_\#[\cdot,\cdot,\cdot]$.

\begin{thm}\label{thm:repaux}
Let $m\in {\mathbb{N}}\setminus\{0\}$, $\alpha\in]0,1[$. Let $\rho>0$. Let \eqref{ass}--\eqref{assf} hold. Let $\epsilon_1$, $\mathcal{U}$, $\mathcal{V}$, $\Xi$ be as in Proposition \ref{prop:Lambda}. Then the following statements hold.
\begin{enumerate}
\item[(i)] Let $V$ be a bounded open subset of $\mathbb{R}^n$ such that $\mathrm{cl}V \subseteq \mathbb{R}^n \setminus (p+q\mathbb{Z}^n)$. Let $r \in \mathbb{N}$. Then there exist $\epsilon_{\#,2} \in ]0,\epsilon_1]$ and a real analytic map $U_{\#}$ from $]-\epsilon_{\#,2},\epsilon_{\#,2}[\times \mathcal{U}\times \mathcal{V}$ to $C^{r}(\mathrm{cl}V)$ such that the following statements hold.
\begin{enumerate}
\item[(j)] $\mathrm{cl}V \subseteq \mathbb{S}[\Omega_{p,\epsilon}]^{-}$ for all $\epsilon \in ]-\epsilon_{\#,2},\epsilon_{\#,2}[$.
\item[(jj)]
\[
u_\#[\epsilon,g,f](x)=\epsilon^{n-1}U_{\#}[\epsilon,g,f](x)+\Xi[\epsilon,g,f] \qquad \forall x \in \mathrm{cl}V\,,
\]
for all $(\epsilon,g,f) \in ]0,\epsilon_{\#,2}[\times \mathcal{U}\times \mathcal{V}$. Moreover,
\begin{align}
U_{\#}[0,g_0,f_0](x)=&-\int_{\partial \Omega}(DS_{q,n}(x-p))\nu_{\Omega}(s)\tilde{\theta}(s)\,d\sigma_s \nonumber\\=&DS_{q,n}(x-p)\int_{\partial \Omega}\nu_{\Omega}(s)\tilde{u}(s)\,d\sigma_s\nonumber\\
&-DS_{q,n}(x-p)\int_{\partial \Omega}s\frac{\partial \tilde{u}}{\partial \nu_{\Omega}}(s)\,d\sigma_s\qquad \forall x \in \mathrm{cl}V\,,\nonumber
\end{align}
where $\tilde{\theta}$, $\tilde{u}$ are as in Lemma \ref{lem:limeq}.
\end{enumerate}
\item[(ii)] Let $\widetilde{V}$ be a bounded open subset of $\mathbb{R}^n \setminus \mathrm{cl}\Omega$. Then there exist $\tilde{\epsilon}_{\#,2} \in ]0,\epsilon_1]$ and a real analytic map $\widetilde{U}_{\#}$ from $]-\tilde{\epsilon}_{\#,2},\tilde{\epsilon}_{\#,2}[\times \mathcal{U}\times \mathcal{V}$ to $C^{m,\alpha}(\mathrm{cl}\widetilde{V})$ such that the following statements hold.
\begin{enumerate}
\item[(j')] $p+\epsilon\mathrm{cl}\widetilde{V}\subseteq Q\setminus \Omega_{p,\epsilon}$ for all $\epsilon \in ]-\tilde{\epsilon}_{\#,2},\tilde{\epsilon}_{\#,2}[\setminus \{0\}$.
\item[(jj')]
\[
u_\#[\epsilon,g,f](p+\epsilon t)=\widetilde{U}_{\#}[\epsilon,g,f](t)+\Xi[\epsilon,g,f]\qquad \forall t \in \mathrm{cl}\widetilde{V}\,,
\]
for all $(\epsilon,g,f) \in ]0,\tilde{\epsilon}_{\#,2}[\times \mathcal{U}\times \mathcal{V}$. Moreover,
\[
\widetilde{U}_{\#}[0,g_0,f_0](t)=\tilde{u}(t)\qquad \forall t \in \mathrm{cl}\widetilde{V}\,.
\]
where $\tilde{u}$ is as in Lemma \ref{lem:limeq}.
\end{enumerate}
\end{enumerate}
\end{thm}
\begin{proof} We follow the argument of \cite[\S 4]{Mu11}. We first consider statement (i). By taking $\epsilon_{\#,2} \in ]0,\epsilon_1]$ small enough, we can clearly assume that (j) holds. Consider now (jj). By Proposition \ref{prop:Lambda}, if $(\epsilon,g,f) \in ]0,\epsilon_{\#,2}[\times \mathcal{U}\times \mathcal{V}$, we have
\[
\begin{split}
u_\#[\epsilon,g,f](x)=-\epsilon^{n-1}\int_{\partial \Omega}(DS_{q,n}(x-p-\epsilon s))\nu_{\Omega}(s)\Theta[\epsilon,g,f](s)\, d\sigma_s&+\Xi[\epsilon,g,f] \\& \forall x \in \mathrm{cl}V\, .
\end{split}
\]
Thus it is natural to set
\[
U_{\#}[\epsilon,g,f](x)\equiv-\int_{\partial \Omega}(DS_{q,n}(x-p-\epsilon s))\nu_{\Omega}(s)\Theta[\epsilon,g,f](s)\, d\sigma_s \qquad \forall x \in \mathrm{cl}V\, ,
\]
for all $(\epsilon,g,f) \in ]-\epsilon_{\#,2},\epsilon_{\#,2}[\times \mathcal{U}\times \mathcal{V}$. Then Proposition \ref{prop:Lambda}, standard properties of integral  operators with real analytic kernels and with no singularity (cf.~\textit{e.g.}, \cite[\S 4]{LaMu10b}), and classical potential theory (cf.~\textit{e.g.}, Miranda~\cite{Mi65}, Lanza and Rossi \cite[Thm.~3.1]{LaRo04}) imply that $U_{\#}$ is a real analytic map from $]-\epsilon_{\#,2},\epsilon_{\#,2}[\times \mathcal{U}\times \mathcal{V}$ to $C^r(\mathrm{cl}V)$ such that (jj) holds (see also \cite[\S 4]{Mu11}.) \par Consider now (ii). Let $R>0$ be such that $(\mathrm{cl}\widetilde{V}\cup \mathrm{cl} \Omega)\subseteq \mathbb{B}_n(0,R)$. By the continuity of the restriction operator from $C^{m,\alpha}(\mathrm{cl}\mathbb{B}_n(0,R)\setminus \Omega)$ to $C^{m,\alpha}(\mathrm{cl}\widetilde{V})$, it suffices to prove statement (ii) with $\widetilde{V}$ replaced by $\mathbb{B}_n(0,R)\setminus \mathrm{cl}\Omega$. By taking $\tilde{\epsilon}_{\#,2} \in ]0,\epsilon_1]$ small enough, we can assume that 
\[
p+\epsilon \mathrm{cl}\mathbb{B}_n(0,R)\subseteq Q \qquad \forall \epsilon \in ]-\tilde{\epsilon}_{\#,2},\tilde{\epsilon}_{\#,2}[\,.
\] 
If $(\epsilon,g,f) \in ]0,\tilde{\epsilon}_{\#,2}[\times \mathcal{U}\times \mathcal{V}$, a simple computation based on the Theorem of change of variables in integrals shows that
\[
\begin{split}
u_\#[\epsilon,g,f](p+\epsilon t)=&-\int_{\partial\Omega}(DS_{n}(t-s) )\nu_{\Omega}(s)\Theta[\epsilon,g,f](s)\,d\sigma_{s}\\&-\epsilon^{n-1}\int_{\partial\Omega}(DR_n(\epsilon(t-s)) )\nu_{\Omega}(s)\Theta[\epsilon,g,f](s)\,d\sigma_{s} +\Xi[\epsilon,g,f]
\end{split}
\]
for all $t \in \mathrm{cl}\mathbb{B}_n(0,R) \setminus \mathrm{cl}\Omega$. If $(\epsilon,g,f) \in ]-\tilde{\epsilon}_{\#,2},\tilde{\epsilon}_{\#,2}[\times \mathcal{U}\times \mathcal{V}$, classical potential theory implies that the function 
\[
-\int_{\partial\Omega}(DS_{n}(t-s) )\nu_{\Omega}(s)\Theta[\epsilon,g,f](s)\,d\sigma_{s}
\]
of the variable $t \in \mathrm{cl}\mathbb{B}_n(0,R) \setminus \mathrm{cl}\Omega$ admits an extension to $\mathrm{cl}\mathbb{B}_n(0,R) \setminus \Omega$ of class $C^{m,\alpha}(\mathrm{cl}\mathbb{B}_n(0,R) \setminus \Omega)$, which we denote by $w^-[\partial \Omega,\Theta[\epsilon,g,f]]_{|\mathrm{cl}\mathbb{B}_n(0,R) \setminus \Omega}$ (cf.~\textit{e.g.}, Miranda~\cite{Mi65}, Lanza and Rossi \cite[Thm.~3.1]{LaRo04}.) Then classical potential theory and Proposition \ref{prop:Lambda} imply that the map from $]-\tilde{\epsilon}_{\#,2},\tilde{\epsilon}_{\#,2}[\times \mathcal{U}\times \mathcal{V}$ to $C^{m,\alpha}(\mathrm{cl}\mathbb{B}_n(0,R) \setminus \Omega)$ which takes $(\epsilon,g,f)$ to $w^-[\partial \Omega,\Theta[\epsilon,g,f]]_{|\mathrm{cl}\mathbb{B}_n(0,R) \setminus \Omega}$ is real analytic (cf.~\textit{e.g.}, Miranda~\cite{Mi65}, Lanza and Rossi \cite[Thm.~3.1]{LaRo04}.)  Therefore,  if we set
\[
\begin{split}
\widetilde{U}_{\#}[\epsilon,g,f]&(t)\equiv w^-[\partial \Omega,\Theta[\epsilon,g,f]]_{|\mathrm{cl}\mathbb{B}_n(0,R) \setminus \Omega}(t)\\&-\epsilon^{n-1}\int_{\partial\Omega}(DR_n(\epsilon(t-s)) )\nu_{\Omega}(s)\Theta[\epsilon,g,f](s)\,d\sigma_{s}  \quad \forall t \in \mathrm{cl}\mathbb{B}_n(0,R) \setminus \Omega\, ,
\end{split}
\]
for all $(\epsilon,g,f) \in ]-\tilde{\epsilon}_{\#,2},\tilde{\epsilon}_{\#,2}[\times \mathcal{U}\times \mathcal{V}$, Proposition \ref{prop:Lambda}, standard properties of integral  operators with real analytic kernels and with no singularity (cf.~\textit{e.g.}, \cite[\S 4]{LaMu10b}), and classical potential theory imply that $\widetilde{U}_{\#}$ is a real analytic map from $]-\tilde{\epsilon}_{\#,2},\tilde{\epsilon}_{\#,2}[\times \mathcal{U}\times \mathcal{V}$ to $C^{m,\alpha}(\mathrm{cl}\mathbb{B}_n(0,R) \setminus \Omega)$ such that (jj') holds with $\widetilde{V}$ replaced by $\mathbb{B}_n(0,R)\setminus \mathrm{cl}\Omega$ (see also \cite[\S 4]{Mu11}.) Thus the proof is complete.
\end{proof}

Then we analyze the behaviour of the energy integral by means of the following.

\begin{thm}\label{thm:enaux}
Let $m\in {\mathbb{N}}\setminus\{0\}$, $\alpha\in]0,1[$. Let \eqref{ass}--\eqref{assf} hold. Let $\epsilon_1$, $\mathcal{U}$, $\mathcal{V}$ be as in Proposition \ref{prop:Lambda}. Then there exist $\epsilon_{\#,3} \in ]0,\epsilon_1]$ and a real analytic map $G_\#$ from $]-\epsilon_{\#,3},\epsilon_{\#,3}[\times \mathcal{U}\times \mathcal{V}$ to $\mathbb{R}$, such that
\[
\int_{Q \setminus \mathrm{cl}\Omega_{p,\epsilon}}|D_x u_\#[\epsilon,g,f](x)|^2\, dx=\epsilon^{n-2}G_\#[\epsilon,g,f]\, ,
\]
for all $(\epsilon,g,f) \in ]0,\epsilon_{\#,3}[\times \mathcal{U}\times \mathcal{V}$. Moreover,
\[
G_\#[0,g_0,f_0]=\int_{\mathbb{R}^n \setminus \mathrm{cl}\Omega}|D\tilde{u}(t)|^2\, dt\,,
\]
where $\tilde{u}$ is as in Lemma \ref{lem:limeq}.
\end{thm}
\begin{proof} We follow the argument of \cite[\S 4]{Mu11}. Let $(\epsilon,g,f) \in ]0,\epsilon_1[\times \mathcal{U}\times \mathcal{V}$. By the Green Formula and by the periodicity of $u_{\#}[\epsilon,g,f](\cdot)$, we have
\begin{equation}\label{eq:en3}
\begin{split}
\int_{Q \setminus \mathrm{cl}\Omega_{p,\epsilon}}&|D_x u_{\#}[\epsilon,g,f](x)|^2\, dx\\
=-&\epsilon^{n-1}\int_{\partial \Omega}D_xu _{\#}[\epsilon,g,f](p+\epsilon t)\nu_{\Omega}(t)u_{\#}[\epsilon,g,f](p+\epsilon t)\, d\sigma_t\\=-&\epsilon^{n-2}\int_{\partial \Omega}D \bigl(u_{\#}[\epsilon,g,f]\circ(p+\epsilon \mathrm{id}_n)\bigr)(t)\nu_{\Omega}(t)\bigl(g(t)-P_q[f](p+\epsilon t)\bigr)\, d\sigma_t\,,
\end{split}
\end{equation}
where $\mathrm{id}_n$ denotes the identity in $\mathbb{R}^n$. Let $R>0$ be such that $\mathrm{cl}\Omega \subseteq \mathbb{B}_n(0,R)$. By Proposition \ref{prop:Lambda} and Theorem \ref{thm:repaux} (ii), there exist $\epsilon_{\#,3} \in ]0,\epsilon_1]$ and a real analytic map $\widehat{U}_\#$ from $]-\epsilon_{\#,3},\epsilon_{\#,3}[\times \mathcal{U}\times \mathcal{V}$ to $C^{m,\alpha}(\mathrm{cl}\mathbb{B}_n(0,R)\setminus \Omega)$, such that
\[
p+\epsilon \mathrm{cl}(\mathbb{B}_n(0,R)\setminus \mathrm{cl}\Omega) \subseteq Q\setminus \Omega_{p,\epsilon} \qquad \forall \epsilon \in ]-\epsilon_{\#,3},\epsilon_{\#,3}[\setminus \{0\}\, ,
\]
and that
\[
\widehat{U}_\#[\epsilon,g,f](t)=u_\#[\epsilon,g,f]\circ (p+\epsilon \mathrm{id}_n)(t) \quad\forall t \in \mathrm{cl}\mathbb{B}_n(0,R)\setminus \Omega \, ,
\]
for all $(\epsilon,g,f) \in ]0,\epsilon_{\#,3}[\times \mathcal{U}\times \mathcal{V}$, and that
\[
\widehat{U}_\#[0,g_0,f_0](t)=\tilde{u}(t)+\tilde{\xi} \qquad \forall t \in \mathrm{cl}\mathbb{B}_n(0,R)\setminus \Omega\,,
\]
where $\tilde{u}$, $\tilde{\xi}$ are as in Lemma \ref{lem:limeq}. By equality \eqref{eq:en3}, we have
\[
\begin{split}
\int_{Q \setminus \mathrm{cl}\Omega_{p,\epsilon}}&|D_x u_\#[\epsilon,g,f](x)|^2\, dx\\&=-\epsilon^{n-2}\int_{\partial \Omega}D_t \widehat{U}_\#[\epsilon,g,f](t)\nu_{\Omega}(t)\bigl(g(t)-P_q[f](p+\epsilon t)\bigr)\, d\sigma_t\,,
\end{split}
\]
for all $(\epsilon,g,f) \in ]0,\epsilon_{\#,3}[\times \mathcal{U}\times \mathcal{V}$. Thus it is natural to set
\[
G_\#[\epsilon,g,f]\equiv -\int_{\partial \Omega}D_t \widehat{U}_\#[\epsilon,g,f](t)\nu_{\Omega}(t)\bigl(g(t)-P_q[f](p+\epsilon t)\bigr) \, d\sigma_t\,,
\]
for all $(\epsilon,g,f) \in ]-\epsilon_{\#,3},\epsilon_{\#,3}[\times \mathcal{U}\times \mathcal{V}$. Then by continuity of the partial derivatives from $C^{m,\alpha}(\mathrm{cl}\mathbb{B}_n(0,R) \setminus \Omega)$ to $C^{m-1,\alpha}(\mathrm{cl}\mathbb{B}_n(0,R) \setminus \Omega)$, and by continuity of the trace operator on $\partial \Omega$ from $C^{m-1,\alpha}(\mathrm{cl}\mathbb{B}_n(0,R) \setminus \Omega)$ to $C^{m-1,\alpha}(\partial \Omega)$, and by the continuity of the pointwise product in Schauder spaces, and by Lemma \ref{lem:Peps} (ii), and by classical potential theory, we conclude that $G_\#$ is a real analytic map from $]-\epsilon_{\#,3},\epsilon_{\#,3}[\times \mathcal{U}\times \mathcal{V}$ to $\mathbb{R}$ and that the Theorem holds (see also \cite[\S 4]{Mu11}.)
\end{proof}

Finally, we consider the integral of $u_{\#}[\cdot,\cdot,\cdot]$, and we prove the following.

\begin{thm}\label{thm:intaux}
Let $m\in {\mathbb{N}}\setminus\{0\}$, $\alpha\in]0,1[$. Let \eqref{ass}--\eqref{assf} hold. Let $\epsilon_1$, $\mathcal{U}$, $\mathcal{V}$ be as in Proposition \ref{prop:Lambda}. Then there exists a real analytic map $J_\#$ from $]-\epsilon_1,\epsilon_1[\times \mathcal{U}\times \mathcal{V}$ to $\mathbb{R}$, such that
\begin{equation}\label{eq:int1aux}
\int_{Q \setminus \mathrm{cl}\Omega_{p,\epsilon}}u_\#[\epsilon,g,f](x)\, dx=J_\#[\epsilon,g,f]\, ,
\end{equation}
for all $(\epsilon,g,f) \in ]0,\epsilon_1[\times \mathcal{U}\times \mathcal{V}$. Moreover,
\begin{equation}\label{eq:int2aux}
J_\#[0,g_0,f_0]=\tilde{\xi}\mathrm{meas}(Q)\,,
\end{equation}
where $\tilde{\xi}$ is as in Lemma \ref{lem:limeq}.
\end{thm}
\begin{proof} Let $(\epsilon,g,f) \in ]0,\epsilon_1[\times \mathcal{U}\times \mathcal{V}$. Clearly,
\[
\begin{split}
\int_{Q \setminus \mathrm{cl}\Omega_{p,\epsilon}}u_\#[\epsilon,g,f](x)\,dx=&\int_{Q \setminus \mathrm{cl}\Omega_{p,\epsilon}}w^-_q\bigl[\partial \Omega_{p,\epsilon},\Theta[\epsilon,g,f]((\cdot-p)/\epsilon)\bigr](x)\,dx\\
&+\Xi[\epsilon,g,f] \bigl(\mathrm{meas}(Q)-\epsilon^n\mathrm{meas}(\Omega)\bigr)\, ,
\end{split}
\]
where $\mathrm{meas}(Q)$ and $\mathrm{meas}(\Omega)$ denote the $n$-dimensional measure of $Q$ and of $\Omega$, respectively. By equality \eqref{perpot}, we have
\[
\begin{split}
w^-_q&\bigl[\partial \Omega_{p,\epsilon},\Theta[\epsilon,g,f]((\cdot-p)/\epsilon)\bigr](x)\\&=-	\sum_{j=1}^n \frac{\partial}{\partial x_j}v^-_q\bigl[\partial \Omega_{p,\epsilon},\Theta[\epsilon,g,f]((\cdot-p)/\epsilon)(\nu_{\Omega_{p,\epsilon}}(\cdot))_j\bigr](x)\quad \forall x \in \mathrm{cl} Q \setminus \mathrm{cl}\Omega_{p,\epsilon}\,.
\end{split}
\]
Let $j \in \{1,\dots,n\}$. By the Divergence Theorem and the periodicity of the periodic simple layer potential, we have
\[
\begin{split}
\int_{Q \setminus \mathrm{cl}\Omega_{p,\epsilon}}&\frac{\partial}{\partial x_j}v^-_q\bigl[\partial \Omega_{p,\epsilon},\Theta[\epsilon,g,f]((\cdot-p)/\epsilon)(\nu_{\Omega_{p,\epsilon}}(\cdot))_j\bigr](x)\,dx\\
=&\int_{\partial Q}v^-_q\bigl[\partial \Omega_{p,\epsilon},\Theta[\epsilon,g,f]((\cdot-p)/\epsilon)(\nu_{\Omega_{p,\epsilon}}(\cdot))_j\bigr](x)(\nu_{Q}(x))_j \,d\sigma_x\\
&-\int_{\partial \Omega_{p,\epsilon}}v^-_q\bigl[\partial \Omega_{p,\epsilon},\Theta[\epsilon,g,f]((\cdot-p)/\epsilon)(\nu_{\Omega_{p,\epsilon}}(\cdot))_j\bigr](x)(\nu_{\Omega_{p,\epsilon}}(x))_j \,d\sigma_x\\
=&-\epsilon^{n-1}\int_{\partial \Omega}v^-_q\bigl[\partial \Omega_{p,\epsilon},\Theta[\epsilon,g,f]((\cdot-p)/\epsilon)(\nu_{\Omega_{p,\epsilon}}(\cdot))_j\bigr](p+\epsilon t)(\nu_{\Omega}(t))_j \,d\sigma_t\,.
\end{split}
\]
Then we note that
\[
\begin{split}
v^-_q\bigl[\partial \Omega_{p,\epsilon},&\Theta[\epsilon,g,f]((\cdot-p)/\epsilon)(\nu_{\Omega_{p,\epsilon}}(\cdot))_j\bigr](p+\epsilon t)\\
=&\epsilon^{n-1}\int_{\partial \Omega}S_n(\epsilon(t-s))\Theta[\epsilon,g,f](s)(\nu_{\Omega}(s))_j\,d\sigma_s\\
&+\epsilon^{n-1}\int_{\partial \Omega}R_n(\epsilon(t-s))\Theta[\epsilon,g,f](s)(\nu_{\Omega}(s))_j\,d\sigma_s \qquad \forall t \in \partial \Omega\,.
\end{split}
\]
We now observe that if $\epsilon >0$ and $x \in \mathbb{R}^n \setminus \{0\}$ then we have
\begin{equation}\label{n:Sneps}
S_{n}(\epsilon x)=\epsilon^{2-n}S_{n}(x)+\delta_{2,n}\frac{1}{2\pi}\log\epsilon\,.
\end{equation}
Moreover, by the Divergence Theorem, it's immediate to see that
\begin{equation}\label{eq:int3aux}
\begin{split}
 \int_{\partial \Omega}\Bigl(\int_{\partial \Omega}&\Theta[\epsilon,g,f](s)(\nu_{\Omega}(s))_j \,d\sigma_s\Bigr)(\nu_{\Omega}(t))_j\,d\sigma_t \\
 &=\Bigl(\int_{\partial \Omega}\Theta[\epsilon,g,f](s)(\nu_{\Omega}(s))_j \,d\sigma_s\Bigr)\Bigl(\int_{\partial \Omega}(\nu_{\Omega}(t))_j\,d\sigma_t\Bigr)=0\,.
\end{split}
\end{equation}
Hence, by equalities \eqref{n:Sneps} and \eqref{eq:int3aux}, if $(\epsilon,g,f) \in ]0,\epsilon_1[\times \mathcal{U}\times \mathcal{V}$, we have
\[
\begin{split}
\int_{Q \setminus \mathrm{cl}\Omega_{p,\epsilon}}w^-_q&\bigl[\partial \Omega_{p,\epsilon},\Theta[\epsilon,g,f]((\cdot-p)/\epsilon)\bigr](x)\,dx\\
=&\sum_{j=1}^n \epsilon^n\biggl[\int_{\partial \Omega}\Bigl(\int_{\partial \Omega}S_n(t-s)\Theta[\epsilon,g,f](s)(\nu_{\Omega}(s))_j \,d\sigma_s\Bigr)(\nu_{\Omega}(t))_j\,d\sigma_t\\
&+\epsilon^{n-2}\int_{\partial \Omega}\Bigl(\int_{\partial \Omega}R_n(\epsilon(t-s))\Theta[\epsilon,g,f](s)(\nu_{\Omega}(s))_j \,d\sigma_s\Bigr)(\nu_{\Omega}(t))_j\,d\sigma_t\biggr]\,.
\end{split}
\]
Thus we set
\[
\begin{split}
\tilde{J}_\#[\epsilon,g,f]&\equiv\sum_{j=1}^n \biggl[\int_{\partial \Omega}\Bigl(\int_{\partial \Omega}S_n(t-s)\Theta[\epsilon,g,f](s)(\nu_{\Omega}(s))_j \,d\sigma_s\Bigr)(\nu_{\Omega}(t))_j\,d\sigma_t\\
&+\epsilon^{n-2}\int_{\partial \Omega}\Bigl(\int_{\partial \Omega}R_n(\epsilon(t-s))\Theta[\epsilon,g,f](s)(\nu_{\Omega}(s))_j \,d\sigma_s\Bigr)(\nu_{\Omega}(t))_j\,d\sigma_t\biggr]\,,
\end{split}
\]
for all $(\epsilon,g,f) \in ]-\epsilon_1,\epsilon_1[\times \mathcal{U}\times \mathcal{V}$. Clearly, if $(\epsilon,g,f) \in ]0,\epsilon_1[\times \mathcal{U}\times \mathcal{V}$, then 
\[
\int_{Q \setminus \mathrm{cl}\Omega_{p,\epsilon}}w^-_q\bigl[\partial \Omega_{p,\epsilon},\Theta[\epsilon,g,f]((\cdot-p)/\epsilon)\bigr](x)\,dx=\epsilon^n \tilde{J}_\#[\epsilon,g,f] \,.
\]
Then the analyticity of $\Theta$, the continuity of the linear map from $C^{m-1,\alpha}(\partial \Omega)$ to $C^{m,\alpha}(\partial \Omega)$ which takes $f$ to the function $\int_{\partial \Omega}S_n(t-s)f(s)\,d\sigma_s$ of the variable $t \in \partial \Omega$ (cf.~\textit{e.g.}, Miranda~\cite{Mi65}, Lanza and Rossi \cite[Thm.~3.1]{LaRo04}), the continuity of the pointwise product in Schauder spaces, standard properties of integral  operators with real analytic kernels and with no singularity (cf.~\textit{e.g.}, \cite[\S 4]{LaMu10b}), and standard calculus in Banach spaces imply that the map $\tilde{J}_\#$ is real analytic from $]-\epsilon_1,\epsilon_1[\times \mathcal{U}\times \mathcal{V}$ to $\mathbb{R}$. Hence, if we set
\[
J_\#[\epsilon,g,f]\equiv \epsilon^n \tilde{J}_\#[\epsilon,g,f]+\Xi[\epsilon,g,f]\bigl(\mathrm{meas}(Q)-\epsilon^n\mathrm{meas}(\Omega)\bigr)\, 
\]
for all $(\epsilon,g,f)\in ]-\epsilon_1,\epsilon_1[\times \mathcal{U}\times \mathcal{V}$, we immediately deduce that $J_\#$ is a real analytic map from $]-\epsilon_1,\epsilon_1[\times \mathcal{U}\times \mathcal{V}$ to $\mathbb{R}$ such that equalities \eqref{eq:int1aux}, \eqref{eq:int2aux} hold, and thus the proof is complete.
\end{proof}

\section{A functional analytic representation Theorem for the solution of problem \eqref{bvp:Direps} } \label{rep}

In this Section, we deduce by the results of Section \ref{form} for $u_\#[\cdot,\cdot,\cdot]$ the corresponding results for $u[\cdot,\cdot,\cdot]$. By formula \eqref{eq:sum} and by Theorem \ref{thm:repaux}, we immediately deduce the following.

\begin{thm}\label{thm:rep}
Let $m\in {\mathbb{N}}\setminus\{0\}$, $\alpha\in]0,1[$. Let $\rho>0$. Let \eqref{ass}--\eqref{assf} hold. Let $\epsilon_1$, $\mathcal{U}$, $\mathcal{V}$ be as in Proposition \ref{prop:Lambda}. Then the following statements hold.
\begin{enumerate}
\item[(i)] Let $V$ be a bounded open subset of $\mathbb{R}^n$ such that $\mathrm{cl}V \subseteq \mathbb{R}^n \setminus (p+q\mathbb{Z}^n)$. Let $r \in \mathbb{N}$. Then there exist $\epsilon_2 \in ]0,\epsilon_1]$ and a real analytic map $U$ from $]-\epsilon_2,\epsilon_2[\times \mathcal{U}\times \mathcal{V}$ to $C^{r}(\mathrm{cl}V)$ such that the following statements hold.
\begin{enumerate}
\item[(j)] $\mathrm{cl}V \subseteq \mathbb{S}[\Omega_{p,\epsilon}]^{-}$ for all $\epsilon \in ]-\epsilon_2,\epsilon_2[$.
\item[(jj)]
\[
u[\epsilon,g,f](x)=U[\epsilon,g,f](x) +P_q[f](x)\qquad \forall x \in \mathrm{cl}V\,,
\]
for all $(\epsilon,g,f) \in ]0,\epsilon_2[\times \mathcal{U}\times \mathcal{V}$. Moreover,
\[
U[0,g_0,f_0](x)=\tilde{\xi}\qquad \forall x \in \mathrm{cl}V\,,
\]
where $\tilde{\xi}$ is as in Lemma \ref{lem:limeq}.
\end{enumerate}
\item[(ii)] Let $\widetilde{V}$ be a bounded open subset of $\mathbb{R}^n \setminus \mathrm{cl}\Omega$. Then there exist $\tilde{\epsilon}_2 \in ]0,\epsilon_1]$ and a real analytic map $\widetilde{U}$ from $]-\tilde{\epsilon}_2,\tilde{\epsilon}_2[\times \mathcal{U}\times \mathcal{V}$ to $C^{m,\alpha}(\mathrm{cl}\widetilde{V})$ such that the following statements hold.
\begin{enumerate}
\item[(j')] $p+\epsilon\mathrm{cl}\widetilde{V}\subseteq Q\setminus \Omega_{p,\epsilon}$ for all $\epsilon \in ]-\tilde{\epsilon}_2,\tilde{\epsilon}_2[\setminus \{0\}$.
\item[(jj')]
\[
u[\epsilon,g,f](p+\epsilon t)=\widetilde{U}[\epsilon,g,f](t)+P_q[f](p+\epsilon t)\qquad \forall t \in \mathrm{cl}\widetilde{V}\,,
\]
for all $(\epsilon,g,f) \in ]0,\tilde{\epsilon}_2[\times \mathcal{U}\times \mathcal{V}$. Moreover,
\[
\widetilde{U}[0,g_0,f_0](t)=\tilde{u}(t) +\tilde{\xi}\qquad \forall t \in \mathrm{cl}\widetilde{V}\,.
\]
where $\tilde{u}$, $\tilde{\xi}$ are as in Lemma \ref{lem:limeq}.
\end{enumerate}
\end{enumerate}
\end{thm}

As far as the energy integral of the solution is concerned, we have the following.

\begin{thm}\label{thm:en}
Let $m\in {\mathbb{N}}\setminus\{0\}$, $\alpha\in]0,1[$. Let \eqref{ass}--\eqref{assf} hold. Let $\epsilon_1$, $\mathcal{U}$, $\mathcal{V}$ be as in Proposition \ref{prop:Lambda}. Then there exist $\epsilon_3 \in ]0,\epsilon_1]$ and a real analytic map $G$ from $]-\epsilon_3,\epsilon_3[\times \mathcal{U}\times \mathcal{V}$ to $\mathbb{R}$, such that
\begin{equation}\label{eq:en1}
\int_{Q \setminus \mathrm{cl}\Omega_{p,\epsilon}}|D_x u[\epsilon,g,f](x)|^2\, dx=\epsilon^{n-2}G[\epsilon,g,f]+\int_{Q}|D_x P_q[f](x)|^2\, dx\, ,
\end{equation}
for all $(\epsilon,g,f) \in ]0,\epsilon_3[\times \mathcal{U}\times \mathcal{V}$. Moreover, 
\begin{equation}\label{eq:en2}
G[0,g_0,f_0]=\int_{\mathbb{R}^n \setminus \mathrm{cl}\Omega}|D\tilde{u}(t)|^2\, dt\,,
\end{equation}
where $\tilde{u}$ is as in Lemma \ref{lem:limeq}.
\end{thm}
\begin{proof} Let $\epsilon_{\#,3}$, $G_\#$ be as in Theorem \ref{thm:enaux}. If $(\epsilon,g,f) \in ]0,\epsilon_{\#,3}[\times \mathcal{U}\times \mathcal{V}$, then we have
\[
\begin{split}
\int_{Q \setminus \mathrm{cl}\Omega_{p,\epsilon}}&|D_x u[\epsilon,g,f](x)|^2\, dx=\int_{Q \setminus \mathrm{cl}\Omega_{p,\epsilon}}|D_x u_\#[\epsilon,g,f](x)|^2\, dx\\+&2\int_{Q \setminus \mathrm{cl}\Omega_{p,\epsilon}}D_x u_\#[\epsilon,g,f](x)\cdot D_xP_q[f](x)\, dx+\int_{Q \setminus \mathrm{cl}\Omega_{p,\epsilon}}|D_x P_q[f](x)|^2\, dx\,.
\end{split}
\]
By the Divergence Theorem, by the harmonicity of $u_\#[\epsilon,g,f]$, and by the periodicity of $u_\#[\epsilon,g,f]$ and $P_q[f]$, we have
\[
\begin{split}
2\int_{Q \setminus \mathrm{cl}\Omega_{p,\epsilon}}D_x &u_\#[\epsilon,g,f](x)\cdot D_xP_q[f](x)\, dx\\&=-2\epsilon^{n-1}\int_{\partial \Omega}P_q[f](p+\epsilon t)\Bigl(\frac{\partial u_\#[\epsilon,g,f]}{\partial \nu_{\Omega_{p,\epsilon}}}\Bigr)(p+\epsilon t)\, d\sigma_t\\
&=-2\epsilon^{n-2}\int_{\partial \Omega}P_q[f](p+\epsilon t)D\bigl(u_\#[\epsilon,g,f]\circ (p+\epsilon \mathrm{id}_n) \bigr)(t)\nu_{\Omega}(t)\, d\sigma_t\,,
\end{split}
\]
where $\mathrm{id}_n$ denotes the identity map in $\mathbb{R}^n$. Let $R>0$ be such that $\mathrm{cl}\Omega \subseteq \mathbb{B}_n(0,R)$. By Proposition \ref{prop:Lambda} and Theorem \ref{thm:repaux} (ii), there exist $\epsilon_{3} \in ]0,\epsilon_{\#,3}]$ and a real analytic map $\widehat{U}_{\#}$ from $]-\epsilon_{3},\epsilon_{3}[\times \mathcal{U}\times \mathcal{V}$ to $C^{m,\alpha}(\mathrm{cl}\mathbb{B}_n(0,R)\setminus \Omega)$, such that
\[
p+\epsilon \mathrm{cl}(\mathbb{B}_n(0,R)\setminus \mathrm{cl}\Omega) \subseteq Q\setminus \Omega_{p,\epsilon} \qquad \forall \epsilon \in ]-\epsilon_{3},\epsilon_{3}[\setminus \{0\}\, ,
\]
and that
\[
\widehat{U}_{\#}[\epsilon,g,f](t)=u_\#[\epsilon,g,f]\circ (p+\epsilon \mathrm{id}_n)(t) \qquad\forall t \in \mathrm{cl}\mathbb{B}_n(0,R)\setminus \Omega \, ,
\]
for all $(\epsilon,g,f) \in ]0,\epsilon_3[\times \mathcal{U}\times \mathcal{V}$, and that
\[
\widehat{U}_{\#}[0,g_0,f_0](t)=\tilde{u}(t)+\tilde{\xi} \qquad \forall t \in \mathrm{cl}\mathbb{B}_n(0,R)\setminus \Omega\,,
\]
where $\tilde{u}$, $\tilde{\xi}$ are as in Lemma \ref{lem:limeq}. Then we have
\[
\begin{split}
2\int_{Q \setminus \mathrm{cl}\Omega_{p,\epsilon}}D_x &u_\#[\epsilon,g,f](x)\cdot D_xP_q[f](x)\, dx\\
&=-2\epsilon^{n-2}\int_{\partial \Omega}P_q[f](p+\epsilon t)D_t \widehat{U}_\#[\epsilon,g,f](t)\nu_{\Omega}(t)\, d\sigma_t\,.
\end{split}
\]
for all $(\epsilon,g,f) \in ]0,\epsilon_3[\times \mathcal{U}\times \mathcal{V}$.
Thus it is natural to set
\[
G_1[\epsilon,g,f]\equiv -2\int_{\partial \Omega}P_q[f](p+\epsilon t)D_t \widehat{U}_\#[\epsilon,g,f](t)\nu_{\Omega}(t)\, d\sigma_t\,,
\]
for all $(\epsilon,g,f) \in ]-\epsilon_{3},\epsilon_{3}[\times \mathcal{U}\times \mathcal{V}$. Then by continuity of the partial derivatives from $C^{m,\alpha}(\mathrm{cl}\mathbb{B}_n(0,R) \setminus \Omega)$ to $C^{m-1,\alpha}(\mathrm{cl}\mathbb{B}_n(0,R) \setminus \Omega)$, and by continuity of the trace operator on $\partial \Omega$ from $C^{m-1,\alpha}(\mathrm{cl}\mathbb{B}_n(0,R) \setminus \Omega)$ to $C^{m-1,\alpha}(\partial \Omega)$, and by the continuity of the pointwise product in Schauder spaces, and by Lemma \ref{lem:Peps} (ii), we conclude that $G_1[\cdot,\cdot,\cdot]$ is a real analytic map from $]-\epsilon_{3},\epsilon_{3}[\times \mathcal{U}\times \mathcal{V}$ to $\mathbb{R}$. Moreover, by classical potential theory, we have
\[
G_1[0,g_0,f_0]=-2\int_{\partial \Omega}P_q[f_0](p)\frac{\partial \tilde{u}}{\partial \nu_{\Omega}}(t)\,d\sigma_t=0\,
\]
(see also \cite[\S 4]{Mu11}.) If $(\epsilon,f) \in ]0,\epsilon_0[\times \mathcal{V}$, then clearly
\[
\int_{Q\setminus \mathrm{cl}\Omega_{p,\epsilon}}|D_xP_q[f](x)|^2\, dx=\int_{Q}|D_xP_q[f](x)|^2\, dx-\epsilon^n\int_{\Omega}|D_xP_q[f](p+\epsilon t)|^2\, dt\,.
\]
By standard properties of functions in the Roumieu class, there exists $\rho' \in ]0,\rho]$ such that the map from $C^0_{q,\omega,\rho}(\mathbb{R}^n)$ to $C^0_{q,\omega,\rho'}(\mathbb{R}^n)$ which takes $f$ to $|D_xP_q[f](\cdot)|^2$ is real analytic (cf.~\textit{e.g.}, the proof of Lanza \cite[Prop.~2.25]{La05}.) By arguing as in the proof of Lemma \ref{lem:Peps}, one can show that the map from $]-\epsilon_0,\epsilon_0[\times \mathcal{V}$ to $C^{m,\alpha}(\mathrm{cl}\Omega)$ which takes $(\epsilon,f)$ to the function $|D_x P_q[f](p+\epsilon t)|^2$ of the variable $t \in \mathrm{cl}\Omega$ is real analytic. Then, by the continuity of the linear operator from $C^{m,\alpha}(\mathrm{cl}\Omega)$ to $\mathbb{R}$ which takes $h$ to $\int_{\Omega}h(y)\, dy$, we immediately deduce that the map from $]-\epsilon_0,\epsilon_0[\times \mathcal{V}$ to $\mathbb{R}$ which takes $(\epsilon,f)$ to $\int_{\Omega}|D_xP_q[f](p+\epsilon t)|^2\, dt$ is real analytic. Thus if we set
\[
G[\epsilon,g,f]\equiv G_\#[\epsilon,g,f]+G_1[\epsilon,g,f]-\epsilon^2\int_{\Omega}|D_xP_q[f](p+\epsilon t)|^2\, dt
\]
for all $(\epsilon,g,f) \in ]-\epsilon_3,\epsilon_3[\times \mathcal{U}\times \mathcal{V}$, we deduce that $G$ is a real analytic map from $ ]-\epsilon_3,\epsilon_3[\times \mathcal{U}\times \mathcal{V}$ to $\mathbb{R}$ such that equalities \eqref{eq:en1}, \eqref{eq:en2} hold.
\end{proof}

\begin{rem} We note that the map from $C^0_{q,\omega,\rho}(\mathbb{R}^n)$ to $\mathbb{R}$ which takes $f$ to $\int_{Q}|D_x P_q[f](x)|^2\, dx$ is real analytic. 
\end{rem}

Finally, we consider the integral of $u[\cdot,\cdot,\cdot]$ and we prove the following.

\begin{thm}\label{thm:int}
Let $m\in {\mathbb{N}}\setminus\{0\}$, $\alpha\in]0,1[$. Let \eqref{ass}--\eqref{assf} hold. Let $\epsilon_1$, $\mathcal{U}$, $\mathcal{V}$ be as in Proposition \ref{prop:Lambda}. Then there exists a real analytic map $J$ from $]-\epsilon_1,\epsilon_1[\times \mathcal{U}\times \mathcal{V}$ to $\mathbb{R}$, such that
\[
\int_{Q \setminus \mathrm{cl}\Omega_{p,\epsilon}}u[\epsilon,g,f](x)\, dx=J[\epsilon,g,f]+\int_{Q}P_q[f](x)\, dx\, ,
\]
for all $(\epsilon,g,f) \in ]0,\epsilon_1[\times \mathcal{U}\times \mathcal{V}$. Moreover,
\begin{equation}\label{eq:int2}
J[0,g_0,f_0]=\tilde{\xi}\mathrm{meas}(Q)\,,
\end{equation}
where $\tilde{\xi}$ is as in Lemma \ref{lem:limeq}.
\end{thm}
\begin{proof} Let $J_\#$ be as in Theorem \ref{thm:intaux}. If we set
\[
J[\epsilon,g,f]\equiv J_\#[\epsilon,g,f]-\epsilon^{n}\int_{\Omega}P_q[f](p+\epsilon t)\,dt
\]
for all $(\epsilon,g,f) \in ]-\epsilon_1,\epsilon_1[\times \mathcal{U}\times \mathcal{V}$, then clearly
\[
\int_{Q \setminus \mathrm{cl}\Omega_{p,\epsilon}}u[\epsilon,g,f](x)\, dx=J[\epsilon,g,f]+\int_{Q}P_q[f](x)\, dx\, ,
\]
for all $(\epsilon,g,f) \in ]0,\epsilon_1[\times \mathcal{U}\times \mathcal{V}$. By Lemma \ref{lem:Peps} (i) and by the continuity of the linear operator from $C^{m,\alpha}(\mathrm{cl}\Omega)$ to $\mathbb{R}$ which takes $h$ to $\int_{\Omega}h(y)\, dy$, we immediately deduce that $J$ is a real analytic map from $]-\epsilon_1,\epsilon_1[\times \mathcal{U}\times \mathcal{V}$ to $\mathbb{R}$ and that \eqref{eq:int2} holds.
\end{proof}

\begin{rem} We note that the map from $C^0_{q,\omega,\rho}(\mathbb{R}^n)$ to $\mathbb{R}$ which takes $f$ to $\int_{Q}P_q[f](x)\, dx$ is linear and continuous. 
\end{rem}

\begin{rem} 
Let the assumptions of Proposition \ref{prop:Lambda} hold. Let $\tilde{u}$, $\tilde{\xi}$ be as in Lemma \ref{lem:limeq}. We observe that if $V$ is a bounded open subset of $\mathbb{R}^n$ such that $\mathrm{cl}V \subseteq \mathbb{R}^n \setminus (p+q\mathbb{Z}^n)$ and if $r \in \mathbb{N}$, then Theorem \ref{thm:rep} (i) implies that
\[
\lim_{(\epsilon,g,f)\to (0,g_0,f_0)}u[\epsilon,g,f]= \tilde{\xi}+P_q[f_0] \qquad \text{in $C^{r}(\mathrm{cl}V)$}\, ,
\]
for all $n \in \mathbb{N} \setminus \{0,1\}$. Similarly, by Theorem \ref{thm:int}, we have
\[
\lim_{(\epsilon,g,f)\to (0,g_0,f_0)}\int_{Q \setminus \mathrm{cl}\Omega_{p,\epsilon}}u[\epsilon,g,f](x)\, dx=\tilde{\xi}\mathrm{meas}(Q)+\int_{Q}P_q[f_0](x)\, dx\, ,
\]
for all $n \in \mathbb{N} \setminus \{0,1\}$. Instead, Theorem \ref{thm:en} implies that
\[
\lim_{(\epsilon,g,f)\to (0,g_0,f_0)}\int_{Q \setminus \mathrm{cl}\Omega_{p,\epsilon}}|D_x u[\epsilon,g,f](x)|^2\, dx=\int_{Q}|D_x P_q[f_0](x)|^2\, dx\, ,
\]
if $n \in \mathbb{N} \setminus \{0,1,2\}$, whereas
\[
\begin{split}
\lim_{(\epsilon,g,f)\to (0,g_0,f_0)}\int_{Q \setminus \mathrm{cl}\Omega_{p,\epsilon}}&|D_x u[\epsilon,g,f](x)|^2\, dx\\&=\int_{\mathbb{R}^n \setminus \mathrm{cl}\Omega}|D\tilde{u}(t)|^2\, dt+\int_{Q}|D_x P_q[f_0](x)|^2\, dx\, ,
\end{split}
\]
if $n=2$. 
\end{rem}

\appendix

\section{A real analyticity result for a composition operator}

In this Appendix, we introduce a slight variant of   Preciso \cite[Prop.~4.2.16, p.~51]{Pr98}, Preciso 
\cite[Prop.~1.1, p.~101]{Pr00} on the real analyticity of a composition operator. See also Lanza \cite[Prop.~2.17, Rem.~2.19]{La98} and the slight variant of the argument of Preciso of the proof of
Lanza \cite[Prop.~9, p.~214]{La07b}.

\begin{prop}\label{prop:Pr}
Let $m$, $h$, $k \in \mathbb{N}$, $h$, $k \geq 1$. Let $\alpha \in ]0,1]$, $\rho>0$. Let $\Omega$, $\Omega '$ be bounded open connected  subsets of $\mathbb{R}^h$, $\mathbb{R}^k$, respectively. Let $\Omega '$ be of class $C^1$. Then the operator $T$ defined by
\[
T[u,v]\equiv u\circ v
\]
for all $(u,v)\in C^0_{\omega,\rho}(\mathrm{cl} \Omega) \times C^{m,\alpha}(\mathrm{cl} \Omega ',\Omega)$ is real analytic from the open subset $C^0_{\omega, \rho}(\mathrm{cl} \Omega)\times C^{m,\alpha}(\mathrm{cl} \Omega', \Omega)$ of $C^0_{\omega,\rho}(\mathrm{cl} \Omega)\times C^{m,\alpha}(\mathrm{cl} \Omega', \mathbb{R}^
h)$ to $C^{m,\alpha}(\mathrm{cl} \Omega')$.
\end{prop}

\section*{Acknowledgment}
The author is endebted to Prof.~M.~Lanza de Cristoforis for proposing him the problem and for a number of comments which have improved the quality of the paper.

\end{document}